\documentclass[12pt,twoside]{article}
\usepackage{amsmath,amssymb,amsthm}

\date{}

\begin{document}


\centerline{}

\centerline{}

\centerline {\Large{\bf The Gray Image of   Codes over}}

\centerline{}

\centerline{\Large{\bf  Finite Chain Rings}}

\centerline{}

\centerline{\bf {Somphong Jitman$^a$ and Patanee
Udomkavanich$^b$}}

\centerline{}

\centerline{Department of Mathematics, Faculty of Science,}

\centerline{Chulalongkorn University, }

\centerline{Bangkok, 10330, Thailand}

\centerline{$^a$somphong.c@student.chula.ac.th,
$^b$pattanee.u@chula.ac.th}

\centerline{}

%
%
%
%
%

\newtheorem{Theorem}{\quad Theorem}[section]

\newtheorem{Definition}[Theorem]{\quad Definition}

\newtheorem{Corollary}[Theorem]{\quad Corollary}

\newtheorem{Proposition}[Theorem]{\quad Proposition}

\newtheorem{Lemma}[Theorem]{\quad Lemma}

\newtheorem{Example}[Theorem]{\quad Example}

\begin{abstract}
  The results of J. F. Qian  et al. \cite{QiMa2008} on
$(1-\gamma)$-constacyclic codes over finite chain rings of
nilpotency index $2$ are extended to $(1-\gamma^e)$-constacyclic
codes over finite chain rings of arbitrary nilpotency index $e+1$.
The Gray  map is introduced for this type of rings. We prove that
the Gray image of a linear $(1 - \gamma^{e })$-constacyclic code
over a finite chain ring is a distance-invariant quasi-cyclic code
over its residue field. When the length of    codes    and the
characteristic of a ring  are relatively prime,  the Gray images
of  a linear cyclic code and a linear $(1+\gamma^e)$-constacyclic
code are permutatively to quasi-cyclic codes over its residue
field.
\end{abstract}



\section{Introduction}

Recently,  J. F.  Qian et al. \cite{QiMa2008} introduced a
$(1-\gamma )$-constacyclic code    over a finite chain ring~$R$ of
nilpotency index $2$. They defined the Gray map from $R^n$ to
$\mathbb{F}_{p^k}^{p^kn}$ and proved that the Gray image of a
linear $(1-\gamma)$-constacyclic code over $R$ is a distance
invariant quasi-cyclic code over $\mathbb{F}_{p^k}$, the residue
field of $R$. In particular, the Gray image   of a cyclic code  of
length~$n$ over $R$ is permutation-equivalent to a quasi-cyclic
code over $\mathbb{F}_{p^k}$ when $\gcd(n,p)=1$.

Motivated by their work, we generalize  their results to the case
of $(1-\gamma^e)$-constacyclic  codes,  cyclic codes  and
$(1+\gamma^e)$-constacyclic  codes over finite chain rings of any
nilpotency index $e+1$, where $e\geq 2$.

 We introduce  a  $(1-\gamma^e)$-constacyclic  code  and  a  $(1+\gamma^e)$-constacyclic  code and  characterize  them in terms of
 corresponding polynomial representation in Section 2.  In
section 3, the Gray   map  on a finite chain ring~$R$  is defined
and the Gray images of a $(1-\gamma^e) $-constacyclic code is
investigated. The special case when the length of codes and the
characteristic of the ring are relatively prime is treated in
section 4.

\section{Preliminaries}\label{sec2}

A finite commutative ring   with identity $1\neq 0$ is called a
\textit{finite chain ring} if its ideals are linearly ordered by
inclusion. It is easily seen  that a   finite chain ring has a
unique maximal ideal. Both the characteristic and the cardinality
of a finite chain ring were shown, in \cite{Di2004}, to be powers
of the characteristic of its residue field. The \textit{nilpotency
index} of a finite chain ring is defined to be the smallest
positive integer $s$ such that $\gamma^s=0$ where $\gamma$ is a
generator of its maximal ideal. This $\gamma$ plays a role of
basis of the ring in the following sense:

\begin{Lemma} [\cite{NoSa2000}]\label{2.0}
  Let $R$ be a finite chain ring of nilpotency index $e+1$, $\gamma$ a generator of its maximal ideal and  $\{0\}\subseteq V \subseteq R$   a set of representatives for
the equivalence classes of $R$ under congruence modulo $\gamma$.
Assume that the residue field $R/\langle \gamma\rangle$ is
$\mathbb{F}_{p^k}$ where $p$ is a prime. Then
   \begin{enumerate}
\item for each $r\in R$ there are unique  $a_0(r), a_1(r) ,\dots
a_{e}(r)\in
 V$ such that $$r=a_0(r)+a_1(r)\gamma+\dots+a_{e}(r)\gamma^{e}, $$
\item  $|V| = p^k$, \item $|\langle\gamma^j \rangle |=
p^{k(e+1-j)}$ for $0\leq j\leq e$.
   \end{enumerate}
\end{Lemma}

  As a consequence of Lemma~\ref{2.0}, an element $\mathbf{r}\in {R}^n$ can be
written  uniquely~as
 $$\mathbf{r}={a}_0(\mathbf{r})+  {a}_1(\mathbf{r})\gamma+\dots+ {a}_e(\mathbf{r})\gamma^e ,$$ where
$ {a}_i(\mathbf{r})=(r_{i,0},r_{i,1},\dots, r_{i,n-1})\in V^n$,
for each $0\leq i\leq e$. We denote $$\widetilde{
{a}_i(\mathbf{r})}=(\widetilde{r_{i,0}},\widetilde{r_{i,1}},\dots,\widetilde{
r_{i,n-1}})$$ where $~\widetilde{}~: R\rightarrow
\mathbb{F}_{p^k}^n$ is the canonical map.

 A \textit{code  of length $n$ over a ring} $R$ is a nonempty
subset of $R^n$. A code $C$ is said to be \textit{linear}  if $C$
is a submodule of the $R$-module $R^n$. A code $C$ over $R$ is
called a {\it cyclic code} if $(c_0,c_1,\dots,c_{n-1})\in C$
implies $(c_{n-1},c_0,\dots,c_{n-2})\in C$.

 Let $R$ be a finite chain ring of nilpotency index $e+1$. Define  \mbox{$\nu  : {R}^n\rightarrow
 {R}^n$}~by
 \begin{align*}
\nu((r_0,r_1,\dots,r_{n-1}))=((1-\gamma^e)r_{n-1}
,r_0,\dots,r_{n-2}).
\end{align*}
A code $C$   over   $R$  is called a $(1-\gamma^e)$-\textit{cyclic
code} if   $\nu( {C})= {C}$.  A $(1+\gamma^e)$-\textit{cyclic
code} over $R$ is defined in a similar fashion.

Let $\sigma^{\otimes
p^{ke-1}}:\mathbb{F}_{p^{k}}^{p^{ke}n}\rightarrow\mathbb{F}_{p^k}^{p^{ke}n}$
be defined by
\[ (a^{(0)}\mid
a^{(1)}\mid\dots\mid a^{({p^{ke-1}}-1)})\mapsto(
 {\sigma }(a^{(0)})\mid
 {\sigma }(a^{(1)})\mid\dots,
 {\sigma }(a^{({p^{ke-1}}-1)})),\] where $a^{(i)}\in
\mathbb{F}_{p^k}^{pn}$, $\mid$ is a vector concatenation and  $
{\sigma }:\mathbb{F}_{p^{k}}^{pn}\rightarrow\mathbb{F}_{p^k}^{pn}$
denotes the
 cyclic shift
$$\sigma((c_0,c_1,\dots,c_{pn-1}))=( c_{pn-1} , c_0 ,\dots, c_{pn-2} ).$$
A code $\widetilde{C}$ of length $p^{ke}n$ over $\mathbb{F}_{p^k}$
satisfying $\sigma^{\otimes
p^{ke-1}}(\widetilde{C})=\widetilde{C}$ is called a
\textit{quasi-cyclic code of  index $p^{ke-1}$}.

A linear cyclic code $C$ over a ring is characterized in terms of
corresponding polynomial representation, $$P(C) =\left\{
c_0+c_1X+\dots+c_{n-1}X^{n-1}\mid (c_0,c_1,\dots,c_{n-1})\in
C\right\}.$$
\begin{Proposition}\label{2.1}
A code $C$ of  length $n$ over a finite chain ring  $R$ is a
linear cyclic code if and only if $P(C)$ is an ideal in the
quotient ring $R[X]/\langle X^n - 1\rangle$.
\end{Proposition}
Analogously, the polynomial representations of
$(1-\gamma^e)$-constacyclic  code and $(1+\gamma^e)$-constacyclic
code are ideals of  rings $R[X]/\langle X^n - (1-\gamma^{e })
\rangle$ and \mbox{$R[X]/\langle X^n - (1+\gamma^{e }) \rangle$,}
respectively.
\begin{Proposition}\label{2.2}
A code $C$ of   length $n$ over a finite chain ring $R$ of
\mbox{nilpotency} index $e+1$ is a linear $(1-\gamma^{e })
$-constacyclic code if and only if $P(C)$ is an ideal in the
quotient ring $R[X]/\langle X^n - (1-\gamma^{e }) \rangle$.
\end{Proposition}

\begin{Proposition}\label{2.3}
A code $C$ of   length $n$ over a finite chain ring $R$ of
\mbox{nilpotency} index $e+1$ is a linear $(1+\gamma^{e })
$-constacyclic code if and only if $P(C)$ is an ideal in the
quotient ring $R[X]/\langle X^n - (1+\gamma^{e }) \rangle$.
\end{Proposition}

In this paper, we work over finite chain rings. Throughout, $R$
denotes a finite chain ring of nilpotency index $e+1$ where
$e\geq2$ and the maximal ideal of $R$ is generated by $\gamma$.
The residue field $R/\langle \gamma\rangle$ is regarded as
$\mathbb{F}_{p^k}$ where $p$ is a prime.

A homogeneous distance on $R^n$ is defined, in \cite{GrSc1999}, in
terms of the weight function $w_{hom}(\textbf{r})$ defined  as
follows:
\[w_{hom}(\textbf{r})=\sum_{i=0}^{n-1}w_{hom}(r_i) \]
  for all  $\textbf{ r}=(r_0,r_1,\dots,r_{n-1})\in  {R}^n$, where
\[w_{hom}(r)=\begin{cases}
  p^{k(e-1)}(p^k-1)&\text{ if } r\in  {R}\setminus  {R}\gamma^e,\\
  p^{ke}&\text{ if } r\in {R}\gamma^e \setminus \{0\},\\
  0&\text{ otherwise. }
\end{cases}\]
The \emph{homogeneous distance} $d_{hom}(\textbf{r},\textbf{s})$
between vectors \mbox{$ \textbf{r},\textbf{s}$ in $ {R}^n$} is
defined to be $w_{hom}(\textbf{r}-\textbf{s})$.

\section{Gray Images of $(1-\gamma^{e })$-constacyclic Codes}

The Gray maps, which are defined in each case, have been used as
tools to linked codes over rings and codes over finite fields. For
a finite chain ring $R$, we define the Gay map from $R^n$ to
$\mathbb{F}_{p^k}^{p^{ke}n}$ as the following:

In order to defined the Gray map from $R^n$ to
$\mathbb{F}_{p^k}^{p^{ke}n}$, recall that each element
$\mathbf{r}\in {R}^n$ can be written  uniquely as
 $$\mathbf{r}={a}_0(\mathbf{r})+  {a}_1(\mathbf{r})\gamma+\dots+ {a}_e(\mathbf{r})\gamma^e ,$$ where
$ {a}_i(\mathbf{r})=(r_{i,0},r_{i,1},\dots, r_{i,n-1})\in V^n$,
for every $0\leq i\leq e$.

Let $\alpha$ be a fixed primitive element of~$\mathbb{F}_{p^k}$.
An element $\alpha_\epsilon\in \mathbb{F}_{p^k}$
\mbox{corresponding} to $\epsilon\in\mathbb{Z}_{p^k}$ is  given by
\[\alpha_\epsilon:=\xi_0(\epsilon)+\xi_1(\epsilon)\alpha+\dots+\xi_{k-1}(\epsilon)\alpha^{k-1},\]
if   the $p$-adic representation of $\epsilon$ is
\[\epsilon=\xi_0(\epsilon)+\xi_1(\epsilon)p+\dots+\xi_{k-1}(\epsilon)p^{k-1},\]
 where $\xi_i(\epsilon)\in
\{0,1,\dots,p-1\}$.

Likewise, each $\omega\in\mathbb{Z}_{p^{ke}}$ is viewed uniquely
as the $p^k$-adic representation
\[\omega=\overline{\xi}_0(\omega)+\overline{\xi}_1(\omega)p^k+\dots+\overline{\xi}_{e-1}(\omega)p^{k(e-1)},\]
where $\overline{\xi}_i(\omega)\in\{0,1,\dots,p^k-1\}$, for every
$0\leq i\leq e-1$.

Define  the Gray map $\Phi: {R}^n\rightarrow
\mathbb{F}_{p^k}^{p^{ke}n }$ by
\begin{align*}
\Phi(\mathbf{r}
)=(\mathbf{b}_0,\mathbf{b}_1,\dots,\mathbf{b}_{p^{ke}-1})
\end{align*}
$\text{ for all } \mathbf{r}={a}_0(\mathbf{r})+
{a}_1(\mathbf{r})\gamma+\dots+ {a}_e(\mathbf{r})\gamma^e\in
{R}^n,$ where
\begin{align}\label{2}
\mathbf{b}_{\omega p^k+\epsilon}=\alpha_\epsilon
\widetilde{{a}_0(\mathbf{r})}+\displaystyle\sum_{l=1}^{e-1}\alpha_{\overline{\xi}_{l-1}(\omega)}
\widetilde{{a}_l(\mathbf{r})} + \widetilde{{a}_e(\mathbf{r}}),
\end{align}
for all $0\leq \omega\leq p^{k(e-1)}-1$ and $0\leq \epsilon \leq
p^k-1$.

 Notice that the Gray map $\Phi$   defined above is
clearly injective. The distance-preserving property is shown in
the next proposition.

\begin{Proposition}\label{inva}
  The Gray map $\Phi$ defined above is an isometry from
  $( {R}^n,d_{hom})$ to
  $(\mathbb{F}_{p^k}^{p^{ke}n},d_{H})$, where $d_H$ denotes the
  Hamming distance on $\mathbb{F}_{p^k}^{p^{ke}n}$.
\end{Proposition}
\begin{proof}
It suffices to show that $w_{hom}(r- {s})=w_H(\Phi( {r })-\Phi( {s
}))$ for all  $ {r }\neq {s}\in {R} $, where $w_H$   the Hamming
weight. We observe that
  \begin{align}
    \Phi( {r}\gamma^e)&=( \widetilde{{a}_0( {r})}, \widetilde{{a}_0( {r})},\dots,  \widetilde{{a}_0( {r})})\label{phi(r)}\\
    \Phi( {r}+ {s}\gamma^e)&=\Phi( {r})+\Phi( {s}\gamma^e).\label{phi(r+s)}
  \end{align}
First, consider the case $ {r }-{s}\in {R}\gamma^e \setminus
\{0\}$.  Then  ${r }-{s}=t\gamma^e$ for some $t\in R$. It follows
from (\ref{phi(r+s)})  that   $ \Phi(r) -\Phi(s) =
\Phi(t\gamma^e)$. Hence, by (\ref{phi(r)}), $w_H(\Phi(r)
-\Phi(s))=w_H(\Phi(t\gamma^e))=p^{ke}=w_{hom}(r-s).$

Next, assume that $ {r }-{s}\in {R}\setminus {R} \gamma^e$. Write
${r}=s+t\gamma^m$, where $0\leq m\leq e-1$   and $ t\in R\setminus
{R}\gamma $. In order to find $w_H(\Phi(r)-\Phi(s))$,  we need to
count the number of $0\leq \omega\leq p^{k(e-1)}-1$ and $0\leq
\epsilon \leq p^k-1$ such that
\begin{align}
  0&=\alpha_\epsilon \left(\widetilde{ {a}_0( r)}-\widetilde{{a}_0(s )}\right)
  +\displaystyle\sum_{l=1}^{e-1}\alpha_{\overline{\xi}_{l-1}(\omega)}
  \left(\widetilde{{a}_l(  r)}-\widetilde{{a}_l(s )}\right) + \left(\widetilde{{a}_e(  r)}-\widetilde{{a}_e(s
  )}\right).\label{lin.eq}
\end{align}

 Since  $ {{a}_i( { t\gamma^m})}=  0$ for
all $0\leq i\leq m-1$, we have ${{a}_i( { r})}=  {{a}_i( { s})}$
for all $0\leq i\leq m-1$ and  the representative of $ {{a}_m( {
r})}- {{a}_m( { s})} $ modulo $ \gamma$ is ${{a}_m( {
t\gamma^m})}.$ It follows from ${{a}_m( { t\gamma^m})}\neq 0$ that
$\widetilde{{a}_m( r)}-\widetilde{{a}_m(s )}\neq
 0$.
Consequently, equation~(\ref{lin.eq}) is a linear
 equation in the $e$ variables $\alpha_\epsilon$ and
 $\alpha_{\overline{\xi}_{l}(\omega)}$ $(0\leq l\leq e-2)$. Hence the number
 of distinct solutions $(\omega,\epsilon)$  is $p^{k(e-1)} $.
Therefore, $w_H(\Phi(r)-\Phi(s))=
p^{ke}-p^{k(e-1)}=p^{k(e-1)}(p^k-1)=w_{hom}(r-s).$
\end{proof}

Next theorem is a main key to characterize
$(1-\gamma^e)$-constacyclic \mbox{codes over~$R$.}
\begin{Theorem}\label{PhioNu}
$ \Phi\circ \nu=\sigma^{\otimes p^{ke-1}}\circ\Phi. $
\end{Theorem}
\begin{proof}
It is more convenient, in this argument, to consider an element
$\mathbf{r}$ as $n$-tuple $\mathbf{r}=(r_0,r_1,\dots,r_{n-1})$,
where $r_i= r_{0,i}+ r_{1,i}\gamma+\dots+r_{e,i}\gamma^e\in {R}$
for $0\leq i\leq n-1$ and to view  $\Phi(\mathbf{r})$ as $
\Phi(\mathbf{r})=(b_0,b_1,\dots,b_{p^{ke}n-1}), $ where
\begin{align*}
b_{(\omega p^k+\epsilon)n+j}=\alpha_\epsilon
\widetilde{r_{0,j}}+\displaystyle\sum_{l=1}^{e-1}\alpha_{\overline{\xi}_{l-1}(\omega)}\widetilde{r_{l,j}}
+\widetilde{r_{e,j}},
\end{align*}
for all $0\leq \omega\leq p^{k(e-1)}-1$,  $0\leq \epsilon \leq
p^k-1$ and $0\leq j\leq n-1$.

Let $ (c_0,c_1,\dots,c_{p^{ke}n-1})=\sigma^{\otimes  {
p^{ke-1}}}\circ\Phi(\mathbf{r}) $. Then for each $0\leq \omega\leq
p^{k(e-1)}-1$, $0\leq \epsilon \leq p^k-1$ and $0\leq j\leq n-1$,
\begin{align*}
c_{(\omega p^k+\epsilon)n+j}   &=\begin{cases}
\alpha_\epsilon \widetilde{r_{0,j-1}}+\displaystyle\sum_{l=1}^{e-1}\alpha_{\overline{\xi}_{l-1}(\omega)}\widetilde{r_{l,j-1}} +\widetilde{r_{e,j-1}} ~~~    \text{ if }j\neq 0  ,\\
(\displaystyle\sum_{i=0}^{k-1}\xi_i(\epsilon)\alpha^i
-1)\widetilde{r_{0,n-1}}+\displaystyle\sum_{l=1}^{e-1}\alpha_{\overline{\xi}_{l-1}(\omega)}\widetilde{r_{l,n-1}}
\\~~~~~~~~~~~~~~~~~~~~~~~~~~~~~+\widetilde{r_{k,n-1}}~~~~~~~~\text{ if }j=0 \text{ and } \xi_0(\epsilon)\neq0 ,\\
(\displaystyle\sum_{i=0}^{k-1}\xi_i(\epsilon)\alpha^i+p-1)\widetilde{r_{0,n-1}}+\displaystyle\sum_{l=1}^{k-1}\alpha_{\overline{\xi}_{l-1}(\omega)}\widetilde{r_{l,n-1}}\\
~~~~~~~~~~~~~~~~~~~~~~~~~~~~~~~~~~+\widetilde{r_{e,n-1}}
 ~~~~\text{ if }j=0 \text{ and } \xi_0(\epsilon)=0.
\end{cases}
\end{align*}
On the other hand, observe that
\begin{eqnarray*}
  \nu(\mathbf{r})  &= &( r_{0,n-1}+ a_{1,n-1}\gamma+\dots+ (r_{e,n-1}-r_{0,n-1})\gamma^e,\\
                            &  & r_{0,0}+ a_{1,0}\gamma+\dots+ r_{e,0}\gamma^e, \dots ,r_{0,n-2}+ r_{1,n-2}\gamma+\dots+
                            r_{e,n-2}\gamma^e).
\end{eqnarray*}
Let  $
  (d_0,d_1,\dots,d_{p^{ke}n-1})= \Phi\circ\nu(\mathbf{r})
$. Then for each $0\leq \omega\leq p^{k(e-1)}-1$,  $0\leq \epsilon
\leq p^k-1$ and $0\leq j\leq n-1$,

\begin{align*}
d_{(\omega p^k+\epsilon)n+j}&=\begin{cases}
\alpha_\epsilon \widetilde{r_{0,j-1}}+\displaystyle\sum_{l=1}^{e-1}\alpha_{\overline{\xi}_{l-1}(\omega)} \widetilde{r_{l,j-1}} +\widetilde{r_{e,j-1} }   &\text{ if }j\neq 0  ,\\
(\alpha_\epsilon-1)\widetilde{r_{0,n-1}}+\displaystyle\sum_{l=1}^{e-1}\alpha_{\overline{\xi}_{l-1}(\omega)}\widetilde{r_{l,n-1}}+\widetilde{r_{e,n-1}}&\text{
if }j=0  ,
\end{cases}\\
&=\begin{cases}
\alpha_\epsilon \widetilde{r_{0,j-1}}+\displaystyle\sum_{l=1}^{e-1}\alpha_{\overline{\xi}_{l-1}(\omega)}\widetilde{r_{l,j-1}} +\widetilde{r_{e,j-1}}   ~~~   \text{ if }j\neq 0  ,\\
(\displaystyle\sum_{i=0}^{k-1}\xi_i(\epsilon)\alpha^i
-1)\widetilde{r_{0,n-1}}+\displaystyle\sum_{l=1}^{e-1}\alpha_{\overline{\xi}_{l-1}(\omega)}\widetilde{r_{l,n-1}}
\\~~~~~~~~~~~~~~~~~~~~~~~~~~~~~+\widetilde{r_{e,n-1}}
 ~~~~~~~~ \text{ if }j=0 \text{ and } \xi_0(\epsilon)\neq0 ,\\
(\displaystyle\sum_{i=0}^{k-1}\xi_i(\epsilon)\alpha^i+p-1)\widetilde{r_{0,n-1}}+\displaystyle\sum_{l=1}^{e-1}\alpha_{\overline{\xi}_{l-1}(\omega)}\widetilde{r_{l,n-1}}
\\~~~~~~~~~~~~~~~~~~~~~~~~~~~~~~~~~~+\widetilde{r_{e,n-1}} ~~~ \text{ if }j=0 \text{ and }
\xi_0(\epsilon)=0 ,
\end{cases}
 \end{align*}
Then the result follows.
\end{proof}

The following theorem gives a characterization of
$(1-\gamma^e)$-constacyclic codes over $R$ in terms of
quasi-cyclic code over $\mathbb{F}_{p^k}$, the residue field of
$R$.
\begin{Theorem}
 A code $C$ of length
  $n$ over $R$  is a $ (1-\gamma^e)$-constacyclic if and only if $\Phi(C)$
  is a    quasi-cyclic  code of index
  $p^{ke-1}$ and length $p^{ke}n$ over $\mathbb{F}_{p^k}$. In
  particular, if $C$ is a linear  $ (1-\gamma^e)$-constacyclic code, then $\Phi(C)$ is a   distance-invariant  quasi-cyclic  code of index
  $p^{ke-1}$ and length $p^{ke}n$.
\end{Theorem}
\begin{proof}
The necessary  part follows directly from Theorem~\ref{PhioNu}:
\[  \sigma ^{\otimes
p^{ke-1}}\circ\Phi(C)=\Phi\circ\nu(C) =\Phi(C) .\] For the
sufficient part, note that if $\Phi(C)$ is quasi-cyclic, then
\[ \Phi(C)= \sigma ^{\otimes p^{ke-1}}\circ\Phi(C)=
\Phi\circ\nu(C).\] The injectivity of $\Phi$ implies $\nu(C)=C$,
that is $C$ is $(1-\gamma^e)$-constacyclic.

The distance-invariant part for   a  linear code follows from
Proposition~\ref{inva}.
\end{proof}

\section{Gray Images of  Linear Cyclic  Codes and Linear $(1+\gamma^e)$-constacyclic Codes    }

Following \cite{LiBl2002}, J. F. Qian et al. proved in
\cite{QiMa2008} that the Gray image of a linear cyclic code
over~$R$ is quasi-cyclic under the Nechaev permutation if the code
length is relatively prime to the characteristic of $R$. Here, we
generalize this result over a finite chain ring $R$ of nilpotency
index greater than $2$. Moreover, we give a description of the
Gray image of a linear $(1+\gamma^e)$-constacyclic code over~$R$.

Throughout, assume that the characteristic of $R$ is
 relatively prime to $n$, the length of codes. Since the characteristic of $R$ is a power of $p$ for some prime $p$, $\gcd(n,p)=1$. Then there exists $n^\prime\in\{ 0, 1, \dots,
p-1\}$ such that $nn^\prime\equiv 1 (\text{mod } p)$. Let
$\beta=1+n^\prime \gamma^e$. Then $\beta^j=( 1+n^\prime
\gamma^e)^j=1+jn^\prime \gamma^e\in R$, for all $j\in\mathbb{Z}$.
In particular, $\beta^n=1+\gamma^e$ and $\beta^{-n}=1-\gamma^e$.

Let $\mu$ be the map defined on the polynomial ring $R[X]$ by
$p(X)\mapsto p(\beta X) $.  Then $\mu$ induces isomorphisms   from
$R[X]/\langle X^n-1\rangle$ to $ R[X]/\langle
  X^n-(1-\gamma^e)\rangle$ and from $R[X]/\langle
X^n-(1+\gamma^e)\rangle$   to $ R[X]/\langle
  X^n-1\rangle$.
Hence $R[X]/\langle X^n-(1+\gamma^e)\rangle$ is isomorphic to $
R[X]/\langle
  X^n-(1-\gamma^e)\rangle$
  via the map
$\mu^2$. These isomorphisms give one-to-one correspondences
between the ideals of these rings.

As \mbox{H. Q. Dinh} et al. showed, in
  \cite{Di2004},  that $R[X]/\langle X^n-1\rangle$ is a principal ideal ring, $R[X]/\langle
  X^n-(1-\gamma^e)\rangle$ and $R[X]/\langle
  X^n+(1-\gamma^e)\rangle$ are also   principal ideal rings.

  Let $ \overline{\mu}: R ^n \rightarrow R ^n $ be a module automorphism defined~by
   \[ (r_0,r_1,\dots,r_{n-1})\mapsto  (r_0,\beta r_1,\dots,\beta^{n-1}r_{n-1}
  ).\]
  Consequently, $\bar{\mu}^2=\bar{\mu}\circ\bar{\mu}$ is a  module automorphism on $R^n$.

Proposition~\ref{cycliF} follows from   Propositions
\ref{2.1}-\ref{2.2} and the definitions of $\mu$
and~$\overline{\mu}$.
\begin{Proposition}\label{cycliF} A code
$C\subseteq R^n$ is a linear cyclic code if and only if
$\overline{\mu}(C)$ is a linear $(1-\gamma^e)$-constacyclic code.
\end{Proposition}

Similarly, Proposition~\ref{+cycliF} follows directly from
Propositions \ref{2.1}, \ref{2.3} and the {definitions} of $\mu^2$
and~$\overline{\mu}^2$.

\begin{Proposition}\label{+cycliF}
A code $C\subseteq R^n$ is a linear $(1+\gamma^e)$-constacyclic
code if and only if $\overline{\mu}^2(C)$ is a linear
$(1+\gamma^e)$-constacyclic code.
\end{Proposition}

In order to characterize the Gray image of a linear cyclic code
over $R$, the Nechaev permutation referred in \cite{QiMa2008} is
extended to be the   permutation  $\tau$   on $\{0,1,\dots,pn-1\}$
defined by
\[\tau(mn+j)=(m+jn^\prime)_pn+j,\]
where $0\leq m\leq p-1$, $0\leq j\leq n-1$, and $(m+jn^\prime)_p$
is the least residue of $m+jn^\prime$ modulo~$p$.  The permutation
$\tau$ is used to define $ {\pi
}:\mathbb{F}_{p^{k}}^{pn}\rightarrow\mathbb{F}_{p^k}^{pn}$ as
follows:
$$\pi((c_0,c_1,\dots,c_{pn-1}))=( c_{\tau(0)} , c_{\tau(1)} ,\dots, c_{\tau(pn-1)}
).$$ The map $\pi$ is then extended to $\pi^{\otimes
{p^{ke-1}}}:\mathbb{F}_{p^{k}}^{p^{ke}n}\rightarrow\mathbb{F}_{p^k}^{p^{ke}n}$
 by
\[ (a^{(p^0)}\mid\dots\mid a^{({p^{ke-1}})})\mapsto(\pi(a^{(p^0)})\mid
\dots\mid \pi(a^{({p^{ke-1}})})),\] where $a^{(i)}\in
\mathbb{F}_{p^k}^{pn}$, $\mid$ is a vector concatenation.

\begin{Proposition}\label{comPi}
   $\Phi\circ\overline{\mu}=\pi^{\otimes
{p^{ke-1}}}\circ\Phi.$
\end{Proposition}
\begin{proof}
First, we conclude  that $
\Phi(\mathbf{r})=(b_0,b_1,\dots,b_{p^{ke}n-1}), $ where
\begin{align*}
b_{(\omega p^k+\epsilon)n+j}=\alpha_\epsilon
\widetilde{r_{0,j}}+\displaystyle\sum_{l=1}^{e-1}\alpha_{\overline{\xi}_{l-1}(\omega)}\widetilde{r_{l,j}}
+\widetilde{r_{e,j}},
\end{align*} for all
$0\leq \omega\leq p^{k(e-1)}-1$,  $0\leq \epsilon \leq p^k-1$ and
$0\leq j\leq n-1$. Let $ (c_0,c_1,\dots,c_{p^{ke}n-1})=
\pi^{\otimes p^{ke-1}}(\Phi(\mathbf{r})).$ Then
\begin{align*}
c_{(\omega p^k+\epsilon)n+j} &=c_{ (\omega p^k+(\xi_0(\epsilon) +\xi_1(\epsilon)p+\dots+\xi_{k-1}(\epsilon)p^{k-1}))n+j}\\
&=c_{ (\omega p^k+(\xi_1(\epsilon)p+\dots+\xi_{k-1}(\epsilon)p^{k-1}))n +\xi_0(\epsilon)n+j}\\
&=b_{ (\omega p^k+(\xi_1(\epsilon)p+\dots+\xi_{k-1}(\epsilon)p^{k-1}))n +(\xi_0(\epsilon)+jn^\prime)_pn+j}\\
&=b_{ (\omega
p^k+((\xi_0(\epsilon)+jn^\prime)_p+\xi_1(\epsilon)p+\dots+\xi_{k-1}(\epsilon)p^{k-1}))n
+j}.
   \end{align*}
On the other hand,   $ \overline{\mu}(\mathbf{r} )=(r_0,\beta
r_1,\dots,\beta^{n-1}r_{n-1})$.  Since  \mbox{ $\beta^j
=1+jn^\prime \gamma^e\in R$,}
\begin{align*}
\beta^j r_j&= (r_{0,j}+ r_{1,j}\gamma+\dots+(jn^\prime
r_{0,j}+r_{e,j})\gamma^e).
\end{align*}
Let $(d_0,d_1,\dots,d_{p^{ke}-1})
=\Phi(\overline{\mu}(\mathbf{r}))$. Then

\begin{align*}
d_{(\omega p^k+\epsilon)n+j} &=  \alpha_\epsilon
\widetilde{r_{0,j}}
+\displaystyle\sum_{l=1}^{e-1}\alpha_{\overline{\xi}_{l-1}(\omega)}
\widetilde{ r_{l,j}} + \widetilde{(jn^\prime   r_{0,j}+r_{e,j} )}\\
 &=  (\alpha_\epsilon +jn^\prime   )\widetilde{r_{0,j}}
+\displaystyle\sum_{l=1}^{e-1}\alpha_{\overline{\xi}_{l-1}(\omega)}
 \widetilde{r_{l,j} }+   \widetilde{ {r_{e,j}}}  \\
  &=  \left((\xi_0(\epsilon)+ \xi_1(\epsilon)\alpha+\dots+\xi_{k-1}(\epsilon)\alpha^{k-1})  +jn^\prime   \right)\widetilde{r_{0,j}}
\\&~~~+\displaystyle\sum_{l=1}^{e-1}\alpha_{\overline{\xi}_{l-1}(\omega)}
 \widetilde{r_{l,j}} +   \widetilde{r_{e,j}}  \\
&=\alpha_{(\xi_0(\epsilon)+jn^\prime)_p+\xi_1(\epsilon)p+\dots+\xi_{k-1}(\epsilon)p^{k-1}} \widetilde{r_{0,j}}+\displaystyle\sum_{l=1}^{e-1}\alpha_{\overline{\xi}_{l-1}(\omega)}\widetilde{r_{l,j}} +\widetilde{r_{e,j}}  \\
&=b_{ (\omega
p^k+((\xi_0(\epsilon)+jn^\prime)_p+\xi_1(\epsilon)p+\dots+\xi_{k-1}(\epsilon)p^{k-1}))n
+j}.
\end{align*}
Then the  theorem is proved.
\end{proof}

Next lemma follows from {Propositions}~\mbox{\ref{cycliF},
\ref{comPi}} and Theorem \ref{PhioNu}.
\begin{Corollary}
  The Gray image of   a linear cyclic code of length $n$ over $R$  is equivalent to a quasi-cyclic
code of index $p^{ke-1}$ and length $p^{ke} n$ over
$\mathbb{F}_{p^k}$.
\end{Corollary}

Finally, we identify the Gray image of a linear
$(1+\gamma^e)$-constacyclic code.

\begin{Proposition}\label{+comPi}
   $\Phi\circ\overline{\mu}^2=\pi^{\otimes
{p^{ke-1}}}\circ\pi^{\otimes {p^{ke-1}}}\circ\Phi.$
\end{Proposition}
\begin{proof}
First, let $ \Phi(\mathbf{r})=(b_0,b_1,\dots,b_{p^{ke}n-1}). $
Then we conclude from the proof of Proposition~\ref{comPi} that  $
 \pi^{\otimes
p^{ke-1}}(\Phi(\mathbf{r}))=(c_0,c_1,\dots,c_{p^{ke}n-1}),$ where
\begin{align*}
c_{(\omega p^k+\epsilon)n+j}  =b_{ (\omega
p^k+((\xi_0(\epsilon)+jn^\prime)_p+\xi_1(\epsilon)p+\dots+\xi_{k-1}(\epsilon)p^{k-1}))n
+j}.
   \end{align*}
 Hence $   \pi^{\otimes
{p^{ke-1}}}( \pi^{\otimes {p^{ke-1}}}
(\Phi(\mathbf{r})))=(d_0,d_1,\dots,d_{p^{ke}n-1}),$ where
\begin{align*}
d_{(\omega p^k+\epsilon)n+j} &=d_{ (\omega p^k+(\xi_0(\epsilon) +\xi_1(\epsilon)p+\dots+\xi_{k-1}(\epsilon)p^{k-1}))n+j}\\
&=d_{ (\omega p^k+(\xi_1(\epsilon)p+\dots+\xi_{k-1}(\epsilon)p^{k-1}))n +\xi_0(\epsilon)n+j}\\
 &=c_{ (\omega p^k+(\xi_1(\epsilon)p+\dots+\xi_{k-1}(\epsilon)p^{k-1}))n +(\xi_0(\epsilon)+jn^\prime)_pn+j}\\
&=c_{ (\omega
p^k+((\xi_0(\epsilon)+jn^\prime)_p+\xi_1(\epsilon)p+\dots+\xi_{k-1}(\epsilon)p^{k-1}))n
+j}\\
&=b_{ (\omega p^k+(\xi_1(\epsilon)p+\dots+\xi_{k-1}(\epsilon)p^{k-1}))n +(\xi_0(\epsilon)+2jn^\prime)_pn+j}\\
&=b_{ (\omega
p^k+((\xi_0(\epsilon)+2jn^\prime)_p+\xi_1(\epsilon)p+\dots+\xi_{k-1}(\epsilon)p^{k-1}))n
+j}.
   \end{align*}
 {On the other hand,   $ \overline{\mu}^2(\mathbf{r}
)=(r_0,\beta^2 r_1,\dots,\beta^{2(n-1)}r_{n-1})$.  Since
$\beta^{2j} =1+2jn^\prime \gamma^e$,}
\begin{align*}
\beta^{2j} r_j&= (r_{0,j}+ r_{1,j}\gamma+\dots+(2jn^\prime
r_{0,j}+r_{e,j})\gamma^e).
\end{align*}
Let $(s_0,s_1,\dots,s_{p^{ke}-1})
=\Phi(\overline{\upsilon}(\mathbf{r}))$. Then
\begin{align*}
s_{(\omega p^k+\epsilon)n+j} &=  \alpha_\epsilon
\widetilde{r_{0,j}}
+\displaystyle\sum_{l=1}^{e-1}\alpha_{\overline{\xi}_{l-1}(\omega)}
\widetilde{ r_{l,j}} + \widetilde{(2jn^\prime   r_{0,j}+r_{e,j} )}\\
 &=  (\alpha_\epsilon +2jn^\prime   )\widetilde{r_{0,j}}
+\displaystyle\sum_{l=1}^{e-1}\alpha_{\overline{\xi}_{l-1}(\omega)}
 \widetilde{r_{l,j} }+   \widetilde{ {r_{e,j}}}  \\
  &=  \left((\xi_0(\epsilon)+ \xi_1(\epsilon)\alpha+\dots+\xi_{k-1}(\epsilon)\alpha^{k-1})  +2jn^\prime   \right)\widetilde{r_{0,j}}
\\
  &~~~~+\displaystyle\sum_{l=1}^{e-1}\alpha_{\overline{\xi}_{l-1}(\omega)}
 \widetilde{r_{l,j}} +   \widetilde{r_{e,j}}  \\
&=\alpha_{(\xi_0(\epsilon)+2jn^\prime)_p+\xi_1(\epsilon)p+\dots+\xi_{k-1}(\epsilon)p^{k-1}} \widetilde{r_{0,j}}+\displaystyle\sum_{l=1}^{e-1}\alpha_{\overline{\xi}_{l-1}(\omega)}\widetilde{r_{l,j}} +\widetilde{r_{e,j}}  \\
&=b_{ (\omega
p^k+((\xi_0(\epsilon)+2jn^\prime)_p+\xi_1(\epsilon)p+\dots+\xi_{k-1}(\epsilon)p^{k-1}))n
+j}.
\end{align*}
The desired result is obtained.
\end{proof}

A delineation  of  the Gray image of a linear
$(1+\gamma^e)$-constacyclic code is given as a consequence of
Propositions~\ref{+cycliF}, \ref{+comPi} and Theorem \ref{PhioNu}.
\begin{Corollary}
  The Gray image of   a linear $(1+\gamma^e)$-constacyclic code of length $n$ over $R$  is equivalent to a quasicyclic
code of index $p^{ke-1}$ and length $p^{ke} n$
over~$\mathbb{F}_{p^k}$.
\end{Corollary}

%
%
%
%
%
%
%
%
%
%


\end{document}